\documentclass[12pt]{amsart}
\usepackage{euscript,epsf,amssymb,xr}

\let\oldmarginpar\marginpar
\renewcommand\marginpar[1]
{\oldmarginpar{\tiny\bf \begin{flushleft} #1 \end{flushleft}}}

\input xy
\xyoption{all}

\newcommand{\la}{\langle}
\newcommand{\ra}{\rangle}

\newtheorem{theorem}{\bf Theorem}[section]
\newtheorem{lemma}[theorem]{\bf Lemma}

\newtheorem{corollary}[theorem]{\bf Corollary}

\newcommand{\CC}{{\Bbb C}}
\newcommand{\RP}{{\Bbb RP}}
\newcommand{\CP}{{\Bbb CP}}

\newcommand{\FF}{{\Bbb F}}

\newcommand{\NN}{{\Bbb N}}

\newcommand{\RR}{{\Bbb R}}

\newcommand{\ZZ}{{\Bbb Z}}

\newcommand{\ggreat}{>\kern-.7ex>}
\newcommand{\ssmall}{<\kern-.7ex<}
\newcommand{\qu}{/\kern-.7ex/}
\newcommand{\exh}{\to\kern-1.8ex\to}

\newcommand{\aA}{{\EuScript{T}_A}}

\newcommand{\cC}{{\EuScript{C}}}

\newcommand{\hH}{{\EuScript{H}}}

\newcommand{\jJ}{{\EuScript{J}}}

\newcommand{\pP}{{\EuScript{P}}}

\newcommand{\tT}{{\EuScript{T}}}

\newcommand{\GL}{\operatorname{GL}}

\newcommand{\Aut}{\operatorname{Aut}}

\newcommand{\Bir}{\operatorname{Bir}}

\newcommand{\Diff}{\operatorname{Diff}}

\newcommand{\Ker}{\operatorname{Ker}}
\newcommand{\Mor}{\operatorname{Mor}}
\renewcommand{\O}{\operatorname{O}}
\newcommand{\sd}{\operatorname{sd}}

\newcommand{\SO}{\operatorname{SO}}
\newcommand{\SU}{\operatorname{SU}}

\newcommand{\Spec}{\operatorname{Spec}}

\newcommand{\vt}{\vartriangleleft}

\title[Finite group actions on homology spheres and manifolds with $\chi\neq 0$]
{Finite group actions on homology spheres and manifolds with nonzero Euler characteristic}

\author{Ignasi Mundet i Riera}
\address{Departament d'\`Algebra i Geometria\\
Facultat de Matem\`atiques\\
Universitat de Barcelona\\
Gran Via de les Corts Catalanes 585\\
08007 Barcelona \\
Spain}
\email{ignasi.mundet@ub.edu}

\date{February 2, 2019}


\thanks{This work has been partially supported by the (Spanish) MEC Projects MTM2012-38122-C03-02
and MTM2015-65361-P}

\begin{document}

\maketitle

\begin{abstract}
Let $X$ be a smooth manifold belonging to one of these three
collections: acyclic manifolds (compact or not, possibly
with boundary), compact connected manifolds (possibly with
boundary) with nonzero Euler characteristic, integral
homology spheres. We prove that $\Diff(X)$ is Jordan. This
means that there exists a constant $C$ such that any finite
subgroup $G$ of $\Diff(X)$ has an abelian subgroup whose index
in $G$ is at most $C$. Using a result of Randall and Petrie we
deduce that the automorphism groups of connected, non necessarily compact,
smooth real affine varieties with nonzero Euler characteristic are Jordan.
\end{abstract}

\section{Introduction}

\subsection{Main result and context}
One of the most fruitful sources of results in finite
transformation groups has been the search for nonlinear
analogues of theorems on linear actions. These include, to
mention a few, results of P.A. Smith \cite[Chap. III]{B},
Dotzel--Hamrick \cite{DH}, tom Dieck \cite{D0} and
Buchdahl--Kwasik--Schultz \cite{BKS}. See also the monographs
\cite{B,Br,D} and the surveys by Davis \cite{MWD} and Schultz
\cite{Sch} for many more examples.

In this paper we prove
nonlinear analogues of the following classical theorem of
Camille Jordan \cite{J}.

\begin{theorem}[Jordan]
\label{thm:Jordan}
For any $n\in\NN$ there exists a constant $C_n$ such that if
$G$ is a finite subgroup of $\GL(n,\RR)$ then $G$ has an
abelian subgroup $A\leq G$ of index at most $C_n$.
\end{theorem}

One can naturally talk about linear actions on spheres, projective spaces, affine spaces
or disks. Our main theorem gives nonlinear analogues of Jordan's theorem in each of these cases:

\begin{theorem}
\label{thm:main}
Let $X$ be a smooth manifold belonging to one of these three collections:
acyclic manifolds (compact or not, possibly with boundary), connected
compact manifolds (possibly with boundary) with nonzero Euler characteristic,
and integral homology spheres.
There exist constants $C,d$ such that any finite group $G$ acting
effectively and smoothly on $X$ has an abelian subgroup $A$ which
can be generated by $d$ elements and has index $[G:A]\leq C$.
Furthermore, $C$ and $d$ can be chosen to depend only on the dimension of $X$
and $H^*(X;\ZZ)$.
\end{theorem}

Although Theorem \ref{thm:main} does not immediately apply to odd dimensional real projective spaces,
combining the statement of Theorem \ref{thm:main} for spheres with the
the arguments in \cite[\S 2.3]{M0} we deduce:

\begin{corollary}
For any $n$ there is a constant $C$ such that any finite group acting
smoothly and effectively on $\RP^n$ has an abelian subgroup of index at most $C$.
\end{corollary}

The proof of Theorem \ref{thm:main} uses a result by Alexandre Turull and the author
\cite{MT} which relies on the classification of finite simple groups (see Section \ref{s:MT}).

In order to explain the context of the results in this paper it will be convenient to use
the following terminology, introduced by Popov \cite{Po0}:
a group $\Gamma$ is said to be Jordan if any finite subgroup $G\leq \Gamma$
has an abelian subgroup whose index in $G$ is bounded above by a constant depending only on $\Gamma$.

Theorem \ref{thm:main} gives a positive partial answer to the
question, posed by \'Etienne Ghys, of whether diffeomorphism
groups of compact manifolds are Jordan (see Question 13.1 in
\cite{F}\footnote{This question was discussed in several talks
by Ghys \cite{G}, but apparently it was in \cite{F} when it
appeared in print for the first time.}).
The particular case in which
$X$ is a sphere was also independently asked in several talks
by Walter Feit and later by Bruno Zimmermann \cite[\S 5]{Z}.

Other partial answers to Ghys's question have been given in the
past. It was been proved that $\Diff(X)$ is Jordan if $X$ is: a
compact manifold of dimension at most $2$ (this is an easy
exercise, see \cite[Theorem 1.3]{M0} for the case of surfaces),
a compact $3$-manifold (Zimmermann, \cite{Z2}), a compact
$4$-manifold with nonzero Euler characteristic \cite{M2}, or a
closed manifold with a nonzero top dimensional cohomology class
expressible as a product of one dimensional classes (e.g. tori)
\cite{M0}. One can also study Jordan's property for
diffeomorphism groups of open manifolds. It is easy to prove
that $\Diff(\RR)$ and $\Diff(\RR^2)$ are Jordan; the work of
Meeks and Yau on minimal surfaces \cite[Theorem 4]{MY} implies
that $\Diff(\RR^3)$ is Jordan; Guazzi, Mecchia and Zimmermann
later proved in \cite{GMZ} that $\Diff(\RR^3)$ and $\Diff(\RR^4)$ are
Jordan using much more elementary methods than those of
\cite{MY}, and Zimmermann proved in \cite{Z2} that
$\Diff(\RR^5)$ and $\Diff(\RR^6)$ are Jordan.

An example of connected open $4$-manifold whose diffeomorphism
group is not Jordan was  given by Popov in \cite{Po}. Csik\'os,
Pyber, and Szab\'o \cite{CPS} found the first example giving a
negative answer to to Ghys's question, showing that the
diffeomorphism group of $T^2\times S^2$ is not Jordan (but note
that, in contrast, the symplectomorphism group of any
symplectic form on $T^2\times S^2$ {\it is} Jordan, see
\cite{M5}). Many other examples of manifolds giving a negative
answer to Ghys's question can
be obtained using the ideas in \cite{CPS}: in particular, for
any manifold $M$ supporting an effective action of $\SU(2)$ or
$\SO(3,\RR)$ the diffeomorphism group of $T^2\times M$ is not
Jordan (see \cite{M6}). It is an intriguing question to
characterize which compact smooth manifolds have Jordan diffeomorphism
group.

\subsection{Jordan property for other automorphism groups}

Let $Z$ be a smooth real affine variety. By an automorphism of $Z$ we understand an automorphism
lifting the identity on $\Spec\RR$. In particular, automorphisms of $Z$ are diffeomorphisms.
Denote by $\Aut(Z)$ the group of automorphisms of $Z$.
Theorem \ref{thm:main} implies that
the group of algebraic automorphisms of acyclic real affine varieties
is Jordan. Combining Theorem \ref{thm:main} with
\cite[Theorem 4.1]{PR} we deduce the following more general result.

\begin{corollary}
\label{cor:affine}
Let $Z$ be a smooth real affine variety. If $Z$ has nonzero Euler characteristic then
$\Aut(Z)$ is Jordan.
\end{corollary}
\begin{proof}
If $Z$ is compact then we apply directly Theorem \ref{thm:main}. Suppose now that
$Z$ is not compact. Let $G\leq\Aut(Z)$ be a finite subgroup. By
\cite[Theorem 4.1]{PR} there exists a smooth compact manifold $X$ with boundary,
endowed with a smooth action of $G$, and a $G$-equivariant diffeomorphism $X\setminus\partial X\cong Z$.
Applying Theorem \ref{thm:main} to $X$ we deduce the existence of an abelian subgroup $A\leq G$
satisfying $[G:A]\leq C$ for some constant $C$ which only depends on the dimension of $X$ and
on $H^*(X;\ZZ)\simeq H^*(Z;\ZZ)$. In particular, $C$ can be chosen as a function of $Z$, and does not
depend on the compactification $X$. This shows that $\Aut(Z)$ is Jordan.
\end{proof}

The question of whether automorphism groups of a given geometric structure is Jordan
has been asked in other contexts besides smooth and smooth real affine varieties.

Serre proved that the classical Cremona group is Jordan
\cite[Theorem 3.1]{S3}, \cite[Theorem 5.3]{S2},
and asked whether higher dimensional
Cremona groups are Jordan \cite[Question 6.1]{S2}. Popov extended the question
to birational transformation groups $\Bir(X)$ and automorphism groups $\Aut(X)$
of general algebraic varieties $X$ (see \cite{Po2} for a survey). Prokhorov and Shramov proved
\cite[Theorem 1.8]{PS} that if $X$ is non-uniruled then $\Bir(X)$ is Jordan and
a recent result of Birkar \cite{Bir} combined with
\cite[Theorem 1.8]{PS} implies that if the irregularity of $X$ is zero then
$\Bir(X)$ is Jordan (in particular the higher Cremona groups
$\Bir(\CP^n)$ are Jordan for every $n$).
Bandman and Zarhin have also recently contributed to this circle of questions
\cite{BZ1,BZ2,BZ3}.
Meng and Zhang proved that the automorphism groups of complex projective manifolds
are always Jordan \cite{MZ} (see \cite{Po3} for some simplifications of their arguments).

Automorphism groups of compact complex surfaces have been
proved to be Jordan by Prokhorov--Shramov \cite{PS}. A recent paper of Popov
\cite{Po3} considers compact complex spaces of arbitrary dimension and gives
some partial results. Automorphism
groups of hyperbolic complex manifolds in particular are also studied in \cite{Po3}.

Ye has proved \cite{Ye} that the group of homeomorphisms of compact
flat manifolds is Jordan. It is an intriguing question whether there exists
a (closed or not) smooth manifold with Jordan diffeomorphism group but with non Jordan homeomorphism
group.

Finally, the Jordan property has also been
explored in symplectic geometry: the author proved that the
Hamiltonian diffeomorphism groups of compact symplectic
manifolds are Jordan, and gave some partial results on full
symplectomorphism groups \cite{M5,M7}.

\subsection{Linear vs. nonlinear actions}
A natural question when considering nonlinear analogues of results
on linear actions is whether for a given $n$ there are finite
groups admitting smooth effective actions on $S^n$ (or more
generally homology $n$-spheres) which are not isomorphic to a
finite subgroup of $\O(n+1,\RR)$. In dimension $3$ the
situation is well understood. Any finite group acting smoothly
and effectively on $S^3$ is isomorphic to a subgroup of
$\SO(4,\RR)$, and in fact any smooth effective action is
conjugate to a linear action, see \cite{DL}. In contrast, some
of the Milnor groups $Q(8a,b,c)$ admit smooth effective actions
on homology $3$-spheres, but admit no effective linear action on $S^3$,
see \cite[Remark 6.11]{DavMil}.
Any finite group acts freely on some rational homology $3$-sphere
\cite{CL}; from this perspective, it would be quite interesting
to extend Theorem \ref{thm:main} to rational homology spheres
(\cite{CL} implies that if their diffeomorphism groups are Jordan then the Jordan
constant can not possibly depend only on the dimension). In dimension $4$ much less is known. Chen, Kwasik
and Schultz \cite{CKS} have recently proved that any finite
group acting smoothly, effectively and orientation preservingly
on $S^4$ is isomorphic to a subgroup of $\SO(5,\RR)$; however,
the analogous question with the orientation preserving
condition removed remains open (but not in the topological
category, see \cite{CKS}).

It is generaly believed that for some $n$ there should exist finite groups,
not isomorphic to any finite subgroup of $\SO(n+1,\RR)$, and admitting
smooth effective and orientation preserving actions on $S^n$ (this is implicit in
the statement of Problem 3 in \cite{CKS}), and very likely there should exist infinitely
many $n$'s for which this happens; similarly, the analogous question for arbitrary
smooth and effective actions, replacing $\SO(n+1,\RR)$ by $\O(n,\RR)$, should admit
examples for infinitely many $n$ (none of such examples can be $p$-groups
by \cite{D0,DH}). However, the author has not found any result along these
lines in the literature (but see some partial results in the opposite direction in
\cite{Z-1,Z-2}).

Some analogues of this question are much better understood.
On the one hand, in the topological category Zimmermann has
recently constructed examples in all dimensions bigger than $5$, see \cite{Z2016}.
On the other hand, it is known that there are infinitely many $n$ for which there is a finite
group $G$ admitting a smooth free action on $S^n$ but no free linear action on the
same sphere (see for example \cite{Pe0}).

One may also wonder whether there exist finite groups acting smoothly and
effectively on $\RR^n$ but admitting no linear effective action on $\RR^n$.
The first nontrivial case, as in the case of spheres, is $n=3$.
The work of Meeks and Yau on minimal surfaces \cite[Theorem 4]{MY} implies that no such groups
exist if $n=3$.
Using much more elementary methods,
Guazzi, Mecchia and Zimmermann proved in \cite{GMZ} that if $n=3$ or $4$ then any finite
group acting smoothly and effectively on a contractible $n$-dimensional manifold is
isomorphic to a subgroup of $\O(n,\RR)$. In general, although one should expect the question
above to have a positive answer for infinitely many values of $n$, no examples seem to
be known.

\subsection{Contents and notation}
Section \ref{s:preliminaries} contains some preliminary results.
In Section \ref{s:MT} we recall the main result in \cite{MT} and prove a slight
strengthening of it. In Section \ref{s:bundles} we consider the situation of an
action of a finite group on a manifold preserving a submanifold on which the
induced action is abelian, and prove a crucial lemma.
The proof of Theorem \ref{thm:main} is given in the last two sections of the paper.
In Section \ref{s:RX} we treat the cases of open acyclic manifolds (Theorem \ref{thm:R-n})
and compact manifolds with nonzero Euler characteristic (Theorem \ref{thm:non-zero-Euler}),
and in Section \ref{s:S} we treat the case of integral homology spheres
(Theorem \ref{thm:spheres}).

For any prime $p$ we denote by $\FF_p$ the field of $p$
elements.

In this paper manifolds or submanifolds need not be connected and may have
connected components of different dimensions. If $X$ is a manifold and $x\in X$
we denote by $\dim_x X$ the dimension of the connected component of $X$ containing
$x$.

\subsection{Remark}
This paper is a substantial revision and slight expansion of the
preprint {\tt arXiv:1403.0383v2}. Version {\tt v3} of that preprint contained
some applications of Theorem \ref{thm:main} on fixed points
which will finally be the content of a separate paper \cite{M8}. When
preparing the present version we included in the introduction a
few references to results on Jordan property that have appeared
after {\tt arXiv:1403.0383v2} was posted, some of which refer
to this paper. We hope these circular citations (which do not include
circular arguments!) will cause no confusion.

\subsection{Acknowledgements}
I am very pleased to acknowledge my indebtedness to Alexandre
Turull. It's thanks to him that Theorem \ref{thm:main}, which
in earlier versions of this paper referred only to finite
solvable groups, has become a theorem on arbitrary finite
groups. I wish also to thank \'Etienne Ghys, Ian Hambleton,
Vladimir Popov and Bruno Zimmermann for useful conversations
and e-mail exchanges. Thanks also to the referee for providing
the reference \cite{Ye2} and for some corrections.

\section{Preliminaries}

\label{s:preliminaries}

\subsection{Local linearization of smooth finite group actions}

The following result is well known. We recall it because of its crucial role
in some of the arguments of this paper. Statement (1)
implies that the fixed point set of a smooth finite group action on a manifold with boundary
is a neat submanifold in the sense of \S 1.4 in \cite{H}.

\begin{lemma}
\label{lemma:linearization}
Let a finite group $\Gamma$ act smoothly on a manifold $X$, and let $x\in X^{\Gamma}$.
The tangent space $T_xX$ carries a linear action of $\Gamma$, defined as the derivative
at $x$ of the action on $X$, satisfying the following properties.
\begin{enumerate}
\item There exist neighborhoods $U\subset T_xX$ and $V\subset X$, of $0\in T_xX$ and $x\in X$ respectively, such that:
    \begin{enumerate}
    \item if $x\notin\partial X$ then there is a $\Gamma$-equivariant diffeomorphism $\phi\colon U\to V$;
    \item if $x\in\partial X$ then there is $\Gamma$-equivariant diffeomorphism $\phi\colon U\cap \{\xi\geq 0\}\to V$, where $\xi$ is a nonzero
        $\Gamma$-invariant element of $(T_xX)^*$ such that $\Ker\xi=T_x\partial X$.
    \end{enumerate}
\item If the action of $\Gamma$ is effective and $X$ is connected then the action of
$\Gamma$ on $T_xX$ is effective, so it induces an inclusion $\Gamma\hookrightarrow\GL(T_xX)$.
\item \label{item:inclusio-propia-subvarietats-fixes}
If $\Gamma'\vt \Gamma$ and $\dim_xX^{\Gamma}<\dim_xX^{\Gamma'}$ then
there exists an irreducible $\Gamma$-submodule $V\subset T_xX$
on which the action of $\Gamma$ is nontrivial but the action of $\Gamma'$ is trivial.
\end{enumerate}
\end{lemma}
\begin{proof}
We first construct a $\Gamma$-invariant Riemannian metric $g$ on $X$ with respect to which $\partial X\subset X$ is totally geodesic. Take any
tangent vector field on a neighborhood of $\partial X$ whose restriction
to $\partial X$ points inward; averaging over the action of $\Gamma$, we get a
$\Gamma$-invariant vector field which still points inward, and its flow at short time
defines an embedding $\psi:\partial X\times[0,\epsilon)\to X$ for some small $\epsilon>0$
such that $\psi(x,0)=x$ and
$\psi(\gamma\cdot x,t)=\gamma\cdot \psi(x,t)$ for any $x\in \partial X$ and $t\in
[0,\epsilon)$. Let $h$ be a $\Gamma$-invariant Riemannian metric on $\partial X$ and consider any Riemannian metric on $X$ whose restriction to $\psi(\partial X\times[0,\epsilon/2])$
is equal to $h+dt^2$. Averaging this metric over the action of $\Gamma$ we obtain a metric $g$
with the desired property. The exponential map with respect to $g$
gives the local diffeomorphism in (1). To prove (2), assume that the action of
$\Gamma$ on $X$ is effective. (1) implies that for any subgroup
$\Gamma'\leq \Gamma$ the fixed point set
$X^{\Gamma'}$ is a submanifold of $X$ and that $\dim_xX=\dim (T_xX)^{\Gamma'}$
for any $x\in X^{\Gamma'}$; furthermore, $X^{\Gamma'}$ is closed by the continuity
of the action. So if some element $\gamma\in\Gamma$
acts trivially on $T_xX$, then $X^{\gamma}$ is a closed
submanifold of $X$ satisfying $\dim_xX^{\gamma}=\dim X$. Since $X$ is connected this implies $X^{\gamma}=X$,
so $\gamma=1$, because the action of $\Gamma$ on $X$ is effective. Finally, (3)
follows from (1) ($V$ can be defined as any of the irreducible factors in the $\Gamma$-module
given by the perpendicular of $T_xX^{\Gamma}$ in $T_xX^{\Gamma'}$).
\end{proof}

\subsection{Points with big stabilizer for actions of $p$-groups on compact manifolds}

The following lemma will be used in the proof that compact connected manifolds with
nonzero Euler characteristic have Jordan diffeomorphism group. The result is
actually true in the wider context of continuous actions of finite $p$-groups
on compact topological manifolds (for example it follows from \cite[Theorem 2.5]{Ye2})
but the smooth case admits a simpler
and basically self contained proof, which we give below for completeness.

\begin{lemma}
\label{lemma:one-big-stabiliser}
Let $Y$ be a compact smooth
manifold, possibly with boundary, satisfying $\chi(Y)\neq 0$. Let $p$ be a prime,
and let $G$ be a finite $p$-group acting smoothly on $Y$. Let $r$ be
the biggest nonnegative integer such that $p^r$ divides $\chi(X)$.
There exists some $y\in Y$ whose stabilizer $G_y$ satisfies $[G:G_y]\leq p^r$.
\end{lemma}

Before explaining the proof we recall some basic facts on equivariant
triangulations.

Let $G$ be a
group and let $\cC$ be a simplicial complex
endowed with an action of $G$. We say that this action is good  if for
any $g\in G$ and any $\sigma\in \cC$ such that $g(\sigma)=\sigma$
we have $g(\sigma')=\sigma'$ for any subsimplex $\sigma'\subseteq\sigma$
(equivalently, the restriction of the action of $g$ to $|\sigma|\subset|\cC|$ is the identity).
This property is called condition (A) in \cite[Chap. III, \S 1]{Br}.
If $\cC$ is a simplicial complex and $G$ acts on
$\cC$, then the induced action of $G$ on the barycentric subdivision
$\sd \cC$ is good (see \cite[Chap. III, Proposition 1.1]{Br}).

Suppose that $G$ acts on a compact manifold $Y$, possibly with boundary.
A $G$-good triangulation of $Y$ is a
pair $(\cC,\phi)$, where $\cC$ is a finite simplicial complex endowed with a
good action of $G$ and $\phi\colon Y\to |\cC|$ is a $G$-equivariant homeomorphism.
For any smooth action of a finite group $G$ on a manifold $X$ there exist
$G$-good triangulations of $X$ (by the previous comments it suffices to prove
the existence of a $G$-equivariant triangulation; this can be easily obtained
adapting the construction of triangulations of smooth manifolds given in \cite{C}
to the finitely equivariant setting; for much more detailed results, see \cite{I}).

We are now ready to prove Lemma \ref{lemma:one-big-stabiliser}.
Let $(\cC,\phi)$ be a $G$-good triangulation of $Y$.
The cardinal of each of the orbits of $G$ acting on $\cC$ is a power of $p$.
If the cardinal of all orbits were divisible by $p^{r+1}$, then
for each $d$ the cardinal of the set of $d$-dimensional simplices in $\cC$
would be divisible by $p^{r+1}$, and consequently
$\chi(Y)=\chi(\cC)$ would also be divisible by $p^{r+1}$, contradicting
the definition of $r$. Hence, there must be at least one simplex $\sigma\in \cC$
whose orbit has at most $p^r$ elements. This means that the stabilizer $G_{\sigma}$ of
$\sigma$ has index at most $p^r$. If $y\in Y$ is a point such that
$\phi(y)\in |\sigma|\subseteq |\cC|$,
then $y$ is fixed by $G_{\sigma}$, because the triangulation is $G$-good.

\subsection{Fixed point loci of actions of abelian $p$-groups}
\label{s:fixed-point-loci-p-groups}

The following lemma is standard, see e.g. \cite[Lemma 5.1]{M7} for the proof.

\begin{lemma}
\label{lemma:cohom-no-augmenta}
Let $X$ be a manifold, let $p$ be a prime, and
let $G$ be a finite $p$-group acting continuously on $X$.
We have
$$\sum_j b_j(X^G;\FF_p)\leq \sum_j b_j(X;\FF_p).$$
\end{lemma}

\section{Testing Jordan's property on $\{p,q\}$-groups}
\label{s:MT}
Suppose that $\cC$ is a set of finite groups. We denote by
$\tT(\cC)$ the set of all $T \in \cC$ such that
there exist primes $p$ and $q$, a Sylow $p$-subgroup $P$ of $T$ (which might
be trivial),
and a normal Sylow $q$-subgroup $Q$ of $T$, such that $T = PQ$.
(In particular, if $T\in \tT(\cC)$ then
$|T|=p^{\alpha}q^{\beta}$
for some primes $p$ and $q$ and nonnegative integers $\alpha,\beta$.)

Let $C$ and $d$ be positive integers.
We say that a set of groups $\cC$ satisfies (the Jordan property) $\jJ(C,d)$
if each $G\in\cC$ has an abelian subgroup $A$ such that $[G:A]\leq C$ and
$A$ can be generated by $d$ elements.
For convenience, we will say that $\cC$ satisfies the Jordan property, without
specifying any constants, whenever there exist some $C$ and $d$ such that
$\cC$ satisfies $\jJ(C,d)$.

The following is the main result in \cite{MT}:

\begin{theorem}
\label{thm:TM}
Let $d$ and $M$ be positive integers.
Let $\cC$ be a set of finite groups which is closed under taking subgroups
and such that $\tT(\cC)$ satisfies $\jJ(M,d)$.
Then there exists a positive integer $C_0$ such that
$\cC$ satisfies $\jJ(C_0,d)$.
\end{theorem}

We next prove a refinement of this theorem.
Given a set of finite groups $\cC$, we define $\aA(\cC)$ exactly like
$\tT(\cC)$ but imposing additionally that the Sylow subgroups are abelian
(in particular, $\aA(\cC)\subseteq\tT(\cC)$).
So a group $G\in\cC$ belongs to $\aA(\cC)$ if and only
if there exist primes $p$ and $q$, an abelian Sylow $p$-subgroup $P\leq G$
and a normal abelian Sylow $q$-subgroup $Q\leq G$, such that $G = PQ$.
Denote by $\pP(\cC)$ the set of groups
$G\in\cC$ with the property that there exists a prime $p$ such that $G$ is a $p$-group.

\begin{corollary}
\label{cor:TM-abelian}
Let $d$ and $M$ be positive integers.
Let $\cC$ be a set of finite groups which is closed under taking subgroups
and such that $\pP(\cC)\cup\aA(\cC)$ satisfies $\jJ(M,d)$.
Then there exists a positive integer $C_0$ such that
$\mathcal C$ satisfies $\jJ(C_0,d)$.
\end{corollary}
\begin{proof}
Let $d$ and $M$ be positive integers, suppose that
$\cC$ is a set of finite groups which closed under taking subgroups,
and assume that $\pP(\cC)\cup\aA(\cC)$ satisfies $\jJ(M,d)$.
Let $C:=M^2(M!)^d$. We claim that $\tT(\cC)$ satisfies $\jJ(C,d)$.
This immediately implies our result, in view of Theorem \ref{thm:TM}.
Since $\pP(\cC)$ satisfies $\jJ(M,d)$,
to justify the claim it suffices to prove the following fact.
\begin{quote}
Let $G$ be a finite
group, let $p,q$ be distinct prime numbers, let $P\leq G$ be a $p$-Sylow subgroup,
let $Q\leq G$ be a normal $q$-Sylow subgroup, and assume that $G=PQ$; if there exist
abelian subgroups $P_0\leq P$ and $Q_0\leq Q$ such that
$[P:P_0]\leq M$, $[Q:Q_0]\leq M$, and $Q_0$ can be generated by $d$
elements, then there exists some $G'\in\aA(G)$ such that $[G:G']\leq C$.
\end{quote}
To prove this, define $Q':=\bigcap_{\phi\in\Aut(Q)}\phi(Q_0)$, where $\Aut(Q)$
denotes the group of automorphisms of $Q$. Clearly $Q'$ is an abelian characteristic subgroup
of $Q$, so it is normal in $G$ (because $Q$ is normal in $G$). Define $G':=P_0Q'$.
Then $G'\in\aA(G)$, so we only need to prove that $[G:G']\leq C$.

Suppose that $\{g_1,\dots,g_\delta\}$ is a generating set of $Q_0$
such that $\delta\leq d$ and
$Q_0\simeq\prod_j\la g_j\ra$, where $\la g_j\ra\leq Q_0$ is the
subgroup generated by $g_j$ (such generating set
exists because $Q_0$ is abelian and can be generated by $d$
elements). If
$\Gamma\leq Q_0$ is any subgroup of index at most $M$, then $g_j^{M!}$ belongs to $\Gamma$ for each $j$.
Consequently, the subgroup $Q''\leq Q_0$ generated by $\{g_1^{M!},\dots,g_{\delta}^{M!}\}$
is contained in any subgroup $\Gamma\leq Q_0$ of index at most $M$. In particular
$Q''\leq Q'$, because for any $\phi\in\Aut(Q)$ we have
$[Q_0:Q_0\cap\phi(Q_0)]\leq M$. On the other hand,
$[Q_0:Q'']\leq (M!)^{\delta}\leq (M!)^d$, so a fortiori
$[Q_0:Q']\leq (M!)^d$. Since $[G:G']=[P:P_0][Q:Q']=[P:P_0][Q:Q_0][Q_0:Q']$, the result follows.
\end{proof}

\section{Actions of finite groups on real vector bundles}
\label{s:bundles}

If $G$ is a (possibly infinite) group we denote by $\cC(G)$ the
set of finite subgroups of $G$ and we let $\aA(G):=\aA(\cC(G))$.

Define, for any smooth manifold $Y$, $$\aA(Y):=\aA(\Diff(Y)).$$
Suppose that $E\to Y$ be a smooth real vector
bundle. Denote by $\Diff(E\to Y)$ the group of
smooth bundle automorphisms lifting arbitrary diffeomorphisms of the base $Y$
(equivalently, diffeomorphisms of $E$ which map
fibers to fibers and whose restriction to each fiber is a linear map).
Denote by
\begin{equation}
\label{eq:map-pi}
\pi:\Diff(E\to Y)\to\Diff(Y)
\end{equation}
the map which assigns to each $\phi\in \Diff(E\to Y)$ the diffeomorphism
of $\psi\in\Diff(Y)$ such that $\phi(E_y)=E_{\psi(y)}$ for every $y$,
where $E_y$ is the fiber of $E$ over $y$.

Define also
$$\aA(E\to Y):=\aA(\Diff(E\to Y)).$$
Since the properties defining the groups in $\aA(\cC)$
in Section \ref{s:MT} are preserved by passing to quotients, for any $G\in\aA(E\to Y$)
we have $\pi(G)\in\aA(Y)$.

\begin{lemma}
\label{lemma:bundle-trick}
Assume that $Y$ is connected and
let $r$ be the rank of $E$. Suppose that $G\in\aA(E\to Y)$ and that
$\pi(G)\in\aA(Y)$ is abelian.
Then $G$ has an abelian subgroup $A\leq G$ satisfying $[G:A]\leq r!$.
\end{lemma}
\begin{proof}
Take a group $G\in\aA(E\to Y)$ such that $\pi(G)$ is abelian. There exist two
distinct primes $p$ and $q$, a $p$-Sylow subgroup $P\leq G$, and a normal
$q$-Sylow subgroup $Q\leq G$, such that $G=PQ$. Furthermore, both $P$ and $Q$ are
abelian.
Let $Q_0=Q\cap\Ker\pi\leq Q$.
Then $Q_0$ is normal in $G$, since $Q_0=Q\cap (G\cap\Ker\pi)$
and both $Q$ and $G\cap\Ker\pi$ are normal in $G$.
So the action of $P$ on $G$ given by conjugation preserves both $Q$ and $Q_0$.
Furthermore, since $\pi(G)$ is abelian we have, for any $\gamma\in P$ and $\eta\in Q$,
$\pi(\gamma\eta\gamma^{-1}\eta^{-1})=\pi(\gamma)\pi(\eta)\pi(\gamma)^{-1}\pi(\eta)^{-1}=1$,
which is equivalent to $\gamma\eta\gamma^{-1}\eta^{-1}\in Q_0$.

The
complexified vector bundle $E\otimes\CC$ splits as a direct sum of subbundles
indexed by the
characters of $Q_0$,
\begin{equation}
\label{eq:descomp-E}
E\otimes\CC=\bigoplus_{\rho\in\Mor(Q_0,\CC^*)}E_{\rho},
\end{equation}
where $v\in E_{\rho}$ if and only if $\eta\cdot v=\rho(\eta)v$ for any $\eta\in Q_0$.
The action of $P$ on $E\otimes\CC$ permutes the summands $\{E_{\rho}\}$. In concrete
terms, if we define for any $\rho\in\Mor(Q_0,\CC^*)$ and $\gamma\in P$ the
character $\rho_{\gamma}\in\Mor(Q_0,\CC^*)$ by $\rho_{\gamma}(\eta)=\rho(\gamma^{-1}\eta\gamma)$ for any $\eta\in Q_0$, then we have
$\gamma\cdot E_{\rho}=E_{\rho_{\gamma}}$. Since there are at most $r$ nonzero summands
in (\ref{eq:descomp-E}) (because $Y$ is connected and the rank of $E$ is $r$),
the subgroup $P'\leq P$ consisting of those elements
which preserve each nonzero subbundle $E_{\rho}$ satisfies $[P:P']\leq r!$.
Furthermore, each element of $P'$ commutes with all the elements in $Q_0$,
because $P'$ acts linearly on $E\otimes\CC$ preserving the summands in
(\ref{eq:descomp-E}) and the action of $Q_0$
on each summand is given by homothecies (in particular, the action of
each element of $Q_0$ lifts the identity on $Y$).
Hence, the action of $P'$ on $Q$ given by conjugation gives a morphism
$$P'\to B:=\{\phi\in\Aut(Q)\mid \phi(\eta)=\eta\text{ for each $\eta\in Q_0$},
\quad \phi(\eta)\eta^{-1}\in Q_0\text{ for each $\eta\in Q$}\}.$$
We now prove that
$B$ is $q$-group. Let $\phi\in B$ be any element, and define a map $f:Q\to Q_0$
as $f(\eta)=\phi(\eta)\eta^{-1}$ for every $\eta$.
Since $Q$ is abelian and $\phi$ is a morphism of groups, $f$
is also a morphism of groups. Furthermore, $f(f(\eta))=1$ for every $\eta\in Q$,
because $Q_0\leq\Ker f$. Using induction it follows that $\phi^k(\eta)=f(\eta)^k\eta$
for every $k\in\NN$. Hence the order of $\phi\in B$ divides
$|Q_0|$, which is a power of $q$.

Since $P'$ is a $p$-group and $p\neq q$, any morphism
$P'\to B$ is trivial. This implies that
$P'$ commutes with $Q$. Setting $A:=P'Q$, the result follows.
\end{proof}

\begin{lemma}
\label{lemma:submfld-trick}
Suppose that $X$ is a smooth connected manifold and that $G\in\aA(X)$. Assume that
there is a $G$-invariant connected submanifold $Y\subseteq X$.
Let $G_Y\leq \Diff(Y)$ be the group
consisting of all diffeomorphisms of $Y$ which are induced by restricting
to $Y$ the action of the elements of $G$. Let $r:=\dim X-\dim Y$.
If $G_Y$ is abelian, then there is an abelian subgroup $A\leq G$ satisfying
$[G:A]\leq r!$.
\end{lemma}
\begin{proof}
There is an inclusion of vector bundles $TY\hookrightarrow TX|_Y$.
Consider the quotient bundle $E:=TX|_Y/TY\to Y$, which is the normal
bundle of the inclusion $Y\hookrightarrow X$.
Let $\Diff(X,Y)$ be the group of diffeomorphisms of $X$ which
preserve $Y$. There is a natural restriction map $\rho:\Diff(X,Y)\to\Diff(E\to Y)$
given by restricting the diffeomorphisms in $\Diff(X,Y)$ to the first
jet of the inclusion $Y\hookrightarrow X$, which gives a bundle automorphism
$TX|_Y\to TX|_Y$ preserving $TY$, and then projecting to an automorphism of $E$.
Furthermore, if $\Gamma\leq\Diff(X,Y)$
is a finite group then by (2) in Lemma \ref{lemma:linearization} the restriction
$\rho|_\Gamma:\Gamma\to\rho(\Gamma)$ is injective.
Applying this to the group $G\in\aA(X)$ in the statement of the lemma we obtain
a group $G_E:=\rho(G)\in\aA(E\to Y)$ which is isomorphic to $G$. Furthermore,
if $\pi:\Diff(E\to Y)\to\Diff(Y)$ is the map (\ref{eq:map-pi}), then
$\pi(G_E)\in\aA(Y)$ coincides with $G_Y$, which by hypothesis is abelian.
We are thus in the setting of Lemma \ref{lemma:bundle-trick}, so we deduce
that $G_E$ (and hence $G$) has an abelian subgroup of index at most $r!$.
\end{proof}

\section{Actions of finite groups on acyclic manifolds and on compact manifolds
with $\chi\neq 0$}
\label{s:RX}

\begin{theorem}
\label{thm:R-n}
For any $n$ there is a constant $C$ such that any finite group acting
smoothly and effectively on an acyclic smooth $n$-dimensional manifold $X$
has an abelian subgroup of index at most $C$.
\end{theorem}

\begin{proof}
Fix $n$ and let $X$ be an acyclic smooth $n$-dimensional manifold.
Let $\cC$ be the set of all finite subgroups of $\Diff(X)$. By Corollary
\ref{cor:TM-abelian} it suffices to prove that $\pP(\cC)\cup\aA(\cC)$ satisfies
the Jordan property. We prove it first for $\pP(\cC)$ and then for $\aA(\cC)$.

Let $G\in\pP(\cC)$ be a finite $p$-group, where $p$ is a prime.
By Smith theory the fixed point set $X^G$ is $\FF_p$-acyclic,
hence nonempty (see \cite[Corollary III.4.6]{B} for the case
$G=\ZZ/p$ and use induction on $|G|$ for the general case, as
in the proof of Lemma \ref{lemma:cohom-no-augmenta} given in \cite{M7}). Let $x\in
X^G$. By (2) in Lemma \ref{lemma:linearization}, linearizing
the action of $G$ at $x$ we get an injective morphism
$G\hookrightarrow \GL(T_xX)\simeq\GL(n,\RR)$. It follows from
Jordan's Theorem \ref{thm:Jordan} that there is an abelian
subgroup $A\leq G$ such that $[G:A]\leq C_n$, where $C_n$
depends only on $n$. Furthermore, since $A$ can be identified
with a subgroup of $\GL(n,\RR)$, it can be generated by at most
$n$ elements. We have thus proved that $\pP(\cC)$ satisfies
$\jJ(C_n,n)$. The same argument also proves that any elementary
$p$-group acting effectively on $X$ has rank at most $n$.

Now let $G\in\aA(\cC)$. By definition, there are two distinct primes $p$ and $q$,
an abelian $p$-Sylow subgroup $P\leq G$,
and a normal abelian $q$-Sylow subgroup $Q\leq G$,
such that $G=PQ$. Let $Y:=X^Q$. By Smith theory, $Y$ is a $\FF_q$-acyclic
manifold (combine again \cite[Corollary III.4.6]{B}
with induction on $|Q|$ as before);
in particular, $Y$ is nonempty and connected. Since $Q$ is normal in $G$,
the action of $G$ on $X$ preserves $Y$. Finally, since the elements of $Q$
act trivially on $Y$, the action of $G$ on $Y$ given by restriction defines an
abelian subgroup of $\Diff(Y)$. This means that we are in the setting of Lemma
\ref{lemma:submfld-trick}, and we can deduce that $G$ has an abelian subgroup
$A\leq G$ of index at most $n!$.
Since, as explained in the previous paragraph, any elementary $p$-group acting on $X$
has rank at most $n$, $A$ can be generated by at most $n$ elements.
We have thus proved that
$\aA(\cC)$ satisfies $\jJ(n!,n)$, and the proof of the theorem is now complete.
\end{proof}

\begin{theorem}
\label{thm:non-zero-Euler}
Let $X$ be a compact connected smooth manifold, possibly with boundary, and satisfying
$\chi(X)\neq 0$. There exists a constant $C$, depending only on the dimension of $X$
and on $H^*(X;\ZZ)$, such that any finite group acting
smoothly and effectively on $X$ has an abelian subgroup of index at most $C$.
\end{theorem}
\begin{proof}
Let $\cC$ be the set of all finite subgroups of $\Diff(X)$. We will again
deduce the theorem from Corollary \ref{cor:TM-abelian}, so we only need to prove
that $\pP(\cC)\cup\aA(\cC)$ satisfies the Jordan property.

To prove that $\pP(\cC)$ satisfies the Jordan property, let $s:=q^e$ be the biggest
prime power dividing $\chi(X)$.
Let $p$ be any prime, and consider a $p$-group $G\in\pP(G)$.
By Lemma \ref{lemma:one-big-stabiliser},
there is a point $x\in X$ whose stabiliser $G_x$ satisfies
$[G:G_x]\leq p^r$, where $p^r$ divides $\chi(X)$. In particular, $[G:G_x]\leq s$.
By Lemma \ref{lemma:linearization} there is an inclusion $G_x\hookrightarrow\GL(T_xX)\simeq\GL(n,\RR)$, where $n=\dim X$. By the same argument
as in the proof of Theorem \ref{thm:R-n}, there is an abelian subgroup $A\leq G_x$
of index $[G_x:A]\leq C_n$ (with $C_n$ depending only on $n$) and which can be generated by $n$ elements. Hence, $\pP(\cC)$ satisfies $\jJ(sC_n,n)$.

We now prove that $\aA(\cC)$ satisfies the Jordan property. Let $G\in\aA(\cC)$.
There exist
two distinct primes $p$ and $q$, an abelian $p$-Sylow subgroup $P\leq G$,
and a normal abelian $q$-Sylow subgroup $Q\leq G$,
such that $G=PQ$. By the arguments used before (involving Lemma \ref{lemma:one-big-stabiliser})
there is a point $x\in X$ whose stabiliser $Q_x$ satisfies
$[Q:Q_x]\leq s$ for some positive integer $s$ depending only on $\chi(X)$.
Since $Q_x$ is abelian and we have an inclusion $Q_x\hookrightarrow\GL(T_xX)$,
we know that $Q_x$ can be generated by $n$ elements.
Define
$$Q':=\bigcap_{\phi\in\Aut(Q)}\phi(Q_x).$$
By the arguments at the end of the proof of Corollary
\ref{cor:TM-abelian},  we have $[Q_x:Q']\leq (s!)^n$.
Since $Q'\leq Q_x$, we have $x\in X^{Q'}$, so $X^{Q'}$ is nonempty.
Also, $Q'$ does not contain any elementary $q$-group of rank greater than
$n$ (because it is a subgroup of $G_x$),
so it can be generated by $n$ elements.
By Lemma \ref{lemma:cohom-no-augmenta} we have
$$\sum_j b_j(X^{Q'};\FF_q)\leq \sum_j b_j(X;\FF_q)\leq K
:=\sum_j\max\{b_j(X;\FF_p)\mid p\text{ prime}\},$$
where $K$ is finite because $X$ is compact.
In particular, $X^{Q'}$ has at most $K$ connected components.
On the other hand,
$Q'$ is a characteristic subgroup of $Q$, and since $Q$ is normal in $G$ it
follows that $Q'$ is also normal in $G$. Hence the action of $P$ on $X$ preserves
$X^{Q'}$. Since the latter has at most $K$ connected components, there exists a
subgroup $P_0\leq P$ of index $[P:P_0]\leq K$ and a connected component
$Y\subseteq X^{Q'}$ such that $P_0$ preserves $Y$.
By Lemma \ref{lemma:one-big-stabiliser}
there is also a subgroup $P'\leq P_0$ which fixes some point in $X$
and such that $[P_0:P']\leq s$, which implies as before
that $P'$ can be generated by at most $n$ elements.
Let $G':=P'Q'$. We can bound
$$[G:G']=[P:P_0][P_0:P'][Q:Q_x][Q_x:Q']\leq sKs(s!)^n.$$
On the other hand, $G'$
preserves $Y$, and its induced action on $Y$ is abelian (because $Q'$ acts
trivially on $Y$ and $P'$ is abelian). By Lemma \ref{lemma:submfld-trick}, there
exists an abelian subgroup $G''\leq G'$ satisfying $[G':G'']\leq n!$.
It follows that
$$[G:G'']=[G:G'][G':G'']\leq M:=sKs(s!)^nn!.$$
Since both $P'$ and $Q'$ can be generated by at most $n$ elements,
$G'$ does not contain any elementary $p$-group or $q$-group of rank
greater than $n$, which implies that
$G''$ can be generated by $n$ elements. We have thus proved that
$\aA(\cC)$ satisfies $\jJ(M,n)$, so the proof of the theorem is complete.
\end{proof}

\section{Actions on integral homology spheres}
\label{s:S}

Recall some standard terminology: given a ring $R$ and an integer $n\geq 0$,
an $R$-homology $n$-sphere is
a topological $n$-manifold $M$ satisfying $H_*(M;R)\simeq H_*(S^n;R)$.
Note that this definition of homology sphere is
more restrictive than that in \cite{B} or \cite{L}, where homology spheres are
not required to be topological manifolds.
An integral homology $n$-sphere is a $\ZZ$-homology $n$-sphere. By
the universal coefficient theorem any integral homology
$n$-sphere is an $\FF_p$-homology $n$-sphere for any prime $p$.
Standard properties of topological manifolds imply that for any
prime $p$ and any integer $n$ any $\FF_p$-homology $n$-sphere
is compact and orientable.

\begin{theorem}
\label{thm:spheres}
For any $n$ there is a constant $C$ such that any finite group acting
smoothly and effectively on a smooth integral homology $n$-sphere has an abelian
subgroup of index at most $C$.
\end{theorem}

Before proving the theorem we collect a few facts which will be
used in our arguments.

For any natural number $n$ we denote by $\Gamma_n$ the
cyclic group of $n$ elements.

Let $p$ be any prime, and let $G$ be a finite $p$-group acting
on an $\FF_p$-homology $n$-sphere $S$. Then the fixed point set
$S^{G}$ is an $\FF_p$-homology $n(G)$-sphere for some integer
$-1\leq n(G)\leq n$, where a $(-1)$-sphere is by convention the
empty set (see \cite[IV.4.3]{B} for the case $G=\Gamma_p$; the
general case follows by induction, see \cite[IV.4.5]{B}).
Furthermore, if $S$ is smooth and the action of $G$ is smooth
and nontrivial, then $S^{G}$ is a smooth proper submanifold of
$S$ and $n(G)<n$.

Suppose that $G\simeq\Gamma_p^r$ is an elementary abelian
$p$-group acting on an $\FF_p$-homology $n$-sphere $S$. The
following formula (\cite[Theorem XIII.2.3]{B}) was proved by
Borel:
\begin{equation}
\label{eq:Borel}
n-n(G)=\sum_{H\leq G \text{ subgroup}\atop [G:H]=p}(n(H)-n(G)).
\end{equation}
Now suppose that $G$ is a finite $p$-group acting on a smooth
$\FF_p$-homology $n$-sphere $S$. Dotzel and Hamrick proved in
\cite{DH} that there exists a real representation of
$\rho:G\to\GL(V)$ such that $\dim V^H=n(H)+1$ for each
subgroup $H\leq G$. This implies that $\dim V=n+1$
(take $H=\{1\}$) and that, if the action is effective, $\rho$
is injective (because for any nontrivial $H\leq G$ we have
$n(H)<n$). Consequently, if $G$ acts effectively on $S$ then we
can identify $G$ with a subgroup of $\GL(n+1,\RR)$ (which, when
$S=S^n$, does not mean that the original action of $G$ on $S^n$
is necessarily linear!). In particular, if $G\simeq\Gamma_p^r$
acts effectively on $S$, then $r\leq n+1$. Alternatively, $r$
can be bounded using a general theorem of Mann and Su
\cite{MS}.

Let $p$ be now an odd prime, let $d$ be a positive integer, and
consider a morphism $\psi:\Gamma_d\to\Aut(\Gamma_p)$.
The following is a very slight modification of a result of
Guazzi and Zimmermann \cite[Lemma 2]{GZ}:

\begin{lemma}
\label{lemma:GZ} If $\Gamma_p \rtimes_{\psi} \Gamma_d$ acts
effectively on a smooth $\FF_p$-homology $n$-sphere $S$ and the
restriction of the action to $\Gamma_p$ is free then all
elements in $\psi(\Gamma_d)$ have order dividing $n+1$.
\end{lemma}

The original result of Guazzi and Zimmermann does not require
the restriction of the action to $\Gamma_p$ to be free, but it
requires the action of $\Gamma_p {\rtimes_{\psi}}\Gamma_d$ to
be orientation preserving. The proof we give of Lemma
\ref{lemma:GZ} is essentially the same as \cite[Lemma 2]{GZ};
we provide details to justify that the result is valid without assuming
that $\Gamma_p {\rtimes_{\psi}}\Gamma_d$ acts
orientation-preservingly.

\begin{proof}
Since $S$ is smooth and compact, $p$ is odd, and the action of
$\Gamma_p$ is free, $n$ must be odd, say $n=2\nu+1$. Let
$\pi:S_{\Gamma_p}\to B\Gamma_p$ be the Borel construction and
let $\zeta:S/\Gamma_p\to B\Gamma_p$ be the composition of a
homotopy equivalence $S/\Gamma_p\simeq S_{\Gamma_p}$ (which
exists because $\ZZ_p$ acts freely on $S$) with $\pi$. A simple
argument using the Serre spectral sequence for $\zeta$ proves
that the map $H^k(\zeta):H^k(B\Gamma_p;\FF_p)\to
H^k(S/\Gamma_p;\FF_p)$ is an isomorphism for $0\leq k\leq n$
(note that, since $p$ is odd, the action of $\Gamma_p$ on $S$
is orientation preserving, so the second page of the Serre
spectral sequence for $\zeta$ has entries
$H^u(B\Gamma_p;\FF_p)\otimes H^v(S;\FF_p)$).

Let $g\in\Gamma_d$ be a generator. There is an integer $y$ such
that $\psi(g)$ sends any $\gamma\in\Gamma_p$ to $\gamma^y$. We
can choose generators $\alpha,\beta$ of $H^*(B\Gamma_p;\FF_p)$
as an $\FF_p$-algebra with $\deg\alpha=1$ (so $\alpha^2=0$),
$\deg\beta=2$, and $\beta=\text{b}(\alpha)$, where $\text{b}$
denotes the Bockstein. In particular, $H^n(B\Gamma_p;\FF_p)$ is
generated by $\alpha\beta^{\nu}$ as an $\FF_p$-vector space.
The action of $\Gamma_d$ on $\Gamma_p$ induces an action
$\phi^*:\Gamma_d\to\Aut (H^*(B\Gamma_p;\FF_p))$ satisfying
$\phi^*(g)(\alpha)=y\alpha$, and by naturality and linearity of
the Bockstein we also have $\phi^*(g)(\beta)=y\beta$. This
implies that the action of $g$ on $H^n(B\Gamma_p;\FF_p)$ is
multiplication by $y^{1+\nu}$. Since the map $\zeta$ is
$\Gamma_d$-equivariant, the action of $g$ on
$H^n(S/\Gamma_p;\FF_p)$ is given by multiplication by
$y^{1+\nu}$. Now, $S$ is compact and orientable and the action
of $\Gamma_p$ is orientation preserving, so $S/\Gamma_p$ is a
compact connected and orientable $n$-manifold. The action of
$g$ on $H^n(S/\Gamma_p;\FF_p)$ is the reduction mod $p$ of the
action on $H^n(S/\Gamma_p;\ZZ)\simeq\ZZ$. Since $g$ acts as a
diffeomorphism on $S/\Gamma_p$, its induced action on
$H^n(S/\Gamma_p;\FF_p)$ is multiplication by $\pm 1$, so it
follows that $y^{1+\nu}\equiv \pm 1\mod p$. Consequently,
$y^{2(1+\nu)}=y^{n+1}\equiv 1\mod p$.
\end{proof}

\subsection{Proof of Theorem \ref{thm:spheres}}
Fix some $n\geq 1$, let $S$ be a smooth integral homology $n$-sphere,
and let $\cC$ be the set of finite subgroups of $\Diff(S)$.
We are going to use Corollary \ref{cor:TM-abelian}, so we need to
prove that $\pP(\cC)\cup\aA(\cC)$ satisfies the Jordan property.
As in the proofs of the other two theorems of this paper, we treat separately
$\pP(\cC)$ and $\aA(\cC)$.

If $G\in\pP(\cC)$ then by the theorem of Dotzel and Hamrick \cite{DH} we may identify
$G$ with a subgroup of $\GL(n+1,\RR)$. By Theorem \ref{thm:Jordan} there is an abelian
subgroup $A\leq G$ of index at most $C_{n+1}$. Furthermore, $A$ can be generated by
$n+1$ elements. Consequently, $\pP(\cC)$ satisfies $\jJ(C_{n+1},n+1)$.

The fact that $\aA(\cC)$ satisfies the Jordan property is a consequence of the following
lemma, combined with the existence of a
uniform upper bound, for any prime $p$, on the rank of elementary
$p$-groups acting effectively on $S^n$ (such bound follows, as we
said, either from the theorem of Dotzel and Hamrick \cite{DH}
or from that of Mann and Su \cite{MS}).

\begin{lemma}
\label{lemma:induction} Given two integers $m\geq 0,r\geq 1$
there exists an integer $K_{m,r}\geq 1$ such that for any two
distinct primes $p$ and $q$, any abelian $p$-group $P$ of rank
at most $r$, any abelian $q$-group $Q$, any morphism
$\phi:P\to\Aut(Q)$, any smooth $\FF_q$-homology $m$-sphere $S$,
and any smooth and effective action of $G:=Q\rtimes_{\phi} P$
on $S$, there is an abelian subgroup $A\leq G$ satisfying
$[G:A]\leq K_{m,r}$.
\end{lemma}
\begin{proof}
Fix some integer $r\geq 1$. We prove the lemma, for this fixed value of $r$,
using induction on $m$. The case $m=0$ being obvious, we may suppose that $m>0$ and
assume that Lemma \ref{lemma:induction} is true for smaller values of $m$.
Let $p,q,P,Q,\phi,G,S$ be as in the statement of the lemma, and take a smooth
effective action of $G$ on $S$.

Suppose first that $S^Q\neq\emptyset$. Then $S^Q$ is a smooth
$\FF_q$-homology sphere of smaller dimension than $S$.
Furthermore, since $Q$ is normal in $G$, the action of $G$ on
$S$ preserves $S^Q$. Let $G_0\leq\Diff(S^Q)$ be the
diffeomorphisms of $S^Q$ induced by restricting the action of
the elements of $G$ on $S$ to $S^Q$. Then $G_0$ is a quotient
of $G$, and this implies that $G_0\simeq Q_0\rtimes P_0$, where
$P_0$ (resp. $Q_0$) is a quotient of $P$ (resp. $Q$). Hence we
may apply the inductive hypothesis to the action of $G_0$ on
$S^Q$ and deduce that there exists an abelian subgroup
$A_0\leq G_0$ satisfying $[G_0:A_0]\leq K_{m-1,r}$. Let
$\pi:G\to G_0$ be the quotient map (i.e., the restriction of
the action to $S^Q$), and let $G':=\pi^{-1}(A_0)\leq G$.
Then $[G:G']\leq K_{m-1,r}$ and the action of $G'$ satisfies
the hypothesis of Lemma \ref{lemma:submfld-trick} with $X=S$
and $Y=S^Q$. Hence, there is an abelian subgroup $A\leq
G'$ satisfying $[G':A]\leq m!$. By the previous estimates $A$
is an abelian subgroup of $G$ of index $[G:A]\leq m!K_{m-1,r}$.

Now assume that $S^Q=\emptyset$. Let $Q':=\{\eta\in Q\mid
\eta^q=1\}\leq Q$. Then $Q'\simeq\Gamma_q^l$. Since $Q'$
is a characteristic subgroup of $Q$ and $Q$ is normal in $G$,
$Q'$ is normal in $G$. We distinguish two cases.

Suppose first that $l\geq 2$. The Borel formula (\ref{eq:Borel}) applied to the
action of $Q'$ on $S$ gives
$$m+1=\sum_{H\leq Q' \text{ subgroup}\atop [Q':H]=p}(n(H)+1).$$
All summands on the RHS are nonnegative integers. So at least one summand
is strictly positive, and there are at most $m+1$ strictly positive summands.
Hence, the set
$$\hH:=\{H\text{ subgroup of }Q'\mid S^H\neq\emptyset,\,[Q':H]=q\}$$
is nonempty and has at most $m+1$ elements. The action of $G$
on $Q'$ by conjugation permutes the elements of $\hH$, so there
is a subgroup $G'\leq G$ fixing some element $H\in\hH$ and
such that $[G:G']\leq m+1$. The fact that $G'$ fixes $H$ as an
element of $\hH$ means that $H$ is a normal subgroup of $G'$,
so the action of $G'$ on $S$ preserves $S^H\neq\emptyset$. We
now proceed along similar lines to the previous case. Let
$G''\leq\Diff(S^H)$ be the diffeomorphisms of $S^H$
induced by restricting the action of the elements of $G'$ on
$S$ to $S^H$. Then $G''$ is a quotient of $G'$ and $S^H$ is a
smooth $\FF_q$-homology sphere of dimension strictly smaller
than $m$; we may thus apply the inductive hypothesis and deduce
the existence of an abelian subgroup $A'\leq G''$ of index
at most $K_{m-1,r}$. Letting $G_a\leq G'$ be the preimage
of $A'$ under the projection map $G'\to G''$ we apply Lemma
\ref{lemma:submfld-trick} to the action of $G_a$ near the
submanifold $S^H\subseteq S$ and conclude that $G_a$ has an
abelian subgroup $A$ of index at most $m!$. Then $[G:A]\leq
(m+1)m!K_{m-1,r}$.

Finally, suppose that $l=1$. In this case $Q'$ acts freely on
$S$ and $Q\simeq\Gamma_{q^s}$. If $q=2$ then
$|\Aut(\Gamma_{q^s})|=(q-1)q^{s-1}$ is equal to $2^{s-1}$ and
since $p\neq 2$ the morphism $\phi:P\to\Aut(Q)$ is trivial,
which means that $G$ is abelian and there is nothing to prove.
Assume then that $q$ is odd. Taking an isomorphism $P\simeq
\Gamma_{p^{e_1}}\times\dots\times \Gamma_{p^{e_c}}$ (with
$c\leq r$, by our assumption on the $p$-rank of $P$) we may
apply Lemma \ref{lemma:GZ} to the restriction of $\phi$ to each
summand, $\phi|_{\Gamma_{p^{e_i}}}:\Gamma_{p^{e_i}}\to
\Aut(\Gamma_q)$ and conclude that there is a subgroup
$G_i\leq \Gamma_{p^{e_i}}$ of index at most $m+1$ such that
$\phi(G_i)$ contains only the trivial automorphism of
$\Gamma_q$, i.e., $G_i$ commutes with $Q'\simeq\Gamma_q$. Let
$P_0:=G_1\times\dots\times G_c$. Then $[P:P_0]\leq (m+1)^c$.
Finally, since, if we identify $\Gamma_q$ with the $q$-torsion
of $\Gamma_{q^s}$, the group
$$\{\alpha\in\Aut(\Gamma_{q^s})\mid \alpha(t)=t\text{ for every $t\in\Gamma_q$}\}$$
is a $q$-group, it follows that $P_0$ not only commutes with
$Q'$, but also with all the elements of $Q$. Consequently
$A:=P_0Q$ is an abelian group, and we have
$$[G:A]\leq (m+1)^c\leq (m+1)^r.$$
This completes the proof of the induction step, and with it that of Lemma \ref{lemma:induction}.
\end{proof}


\begin{thebibliography}{99}

\bibitem{BZ1}
T. Bandman, Y.G. Zarhin,
Jordan groups, conic bundles and abelian varieties,
{\em  Algebr. Geom.} {\bf 4} (2017), no. 2, 229--246.

\bibitem{BZ2}
T. Bandman, Y.G. Zarhin,
Jordan groups and algebraic surfaces,
{\em Transform. Groups} {\bf 20} (2015), no. 2, 327--334.

\bibitem{BZ3}
T. Bandman, Y.G. Zarhin,
Jordan properties of automorphism groups of certain open algebraic varieties,
{\it preprint} {\tt arXiv:1705.07523}.

\bibitem{Bir}
C. Birkar,
Singularities of linear systems and boundedness of Fano varieties,
{\it preprint} {\tt arXiv:1609.05543v1}.

\bibitem{B} A. Borel,
{\em Seminar on transformation groups}, Ann. of Math. Studies {\bf 46},
Princeton University Press, N.J., 1960.

\bibitem{Br}
G.E. Bredon,
{\em Introduction to compact transformation groups},
Pure and Applied Mathematics, Vol. {\bf 46}, Academic Press, New York-London (1972).

\bibitem{BKS}
N.P. Buchdahl, S. Kwasik, R. Schultz,
One fixed point actions on low-dimensional spheres,
{\em Invent. Math.} {\bf 102} (1990), no. 3, 633--662.

\bibitem{C}
S.S. Cairns, A simple triangulation method for smooth manifolds,
{\em Bull. Amer. Math. Soc.} {\bf 67} (1961) 389--390.

\bibitem{CKS}
W. Chen, S. Kwasik, R. Schultz,
Finite symmetries of $S^4$,
{\em Forum Math.} {\bf 28} (2016), no. 2, 295--310.

\bibitem{CL}
D. Cooper, D.D. Long,
Free actions of finite groups on rational homology 3-spheres,
{\em Topology Appl.} {\bf 101} (2000), no. 2, 143--148.

\bibitem{CPS} B. Csik\'os, L. Pyber, E. Szab\'o, Diffeomorphism groups of compact
$4$-manifolds are not always Jordan, preprint {\tt arXiv:1411.7524}.

\bibitem{MWD} M.W. Davis, A survey of results in higher dimensions, Chapter XI in
{\em The Smith conjecture}. Papers presented at the symposium held at Columbia University, New York, 1979. Edited by John W. Morgan and Hyman Bass. Pure and Applied Mathematics, 112. Academic Press, Inc., Orlando, FL, 1984.


\bibitem{DavMil} 
J.F. Davis, R.J. Milgram, {\em A survey of the spherical space form problem}.
Mathematical Reports, 2, Part 2. Harwood Academic Publishers, Chur, 1985.


\bibitem{D0} T. tom Dieck,
Homotopiedarstellungen endlicher Gruppen: Dimensionsfunktionen,
{\em Invent. Math.} {\bf 67} (1982), no. 2, 231--252.

\bibitem{D} T. tom Dieck, {\em Transformation groups},
de Gruyter Studies in Mathematics {\bf 8}, Walter de Gruyter \& Co., Berlin, 1987.

\bibitem{DL}
J. Dinkelbach, B. Leeb,
Equivariant Ricci flow with surgery and applications to finite group actions on
geometric 3-manifolds, {\em Geom. Topol.} {\bf 13} (2009), no. 2, 1129--1173.


\bibitem{DH}
R.M. Dotzel, G.C. Hamrick,
$p$-group actions on homology spheres,
{\em Invent. Math.} {\bf 62} (1981) 437--442.

\bibitem{F}
D. Fisher,
Groups acting on manifolds: around the Zimmer program,
{\em Geometry, rigidity, and group actions}, 72–157,
Chicago Lectures in Math., Univ. Chicago Press, Chicago, IL (2011).

\bibitem{G}
\'E. Ghys, The following talks:
{\em Groups of diffeomorphisms},
Col\'oquio brasileiro de matem\'aticas, Rio de Janeiro (Brasil), July 1997;
{\em The structure of groups acting on manifolds},
Annual meeting of the Royal Mathematical Society, Southampton (UK), March 1999;
{\em Some open problems concerning group actions},
Groups acting on low dimensional manifolds, Les Diablerets (Switzerland), March 2002;
{\em Some Open problems in foliation theory},
Foliations 2006, Tokyo (Japan), September 2006.


\bibitem{GMZ}
A. Guazzi, M. Mecchia, B. Zimmermann, On finite groups acting on acyclic low-
dimensional manifolds, {\em Fund. Math.} {\bf 215} (2011) 203--217.

\bibitem{GZ}
A. Guazzi, B. Zimmermann,
On finite simple groups acting on homology spheres,
Monatshefte für Mathematik {\bf 169} (2013) 371--381.

\bibitem{H}
M.W. Hirsch, {\em Differential Topology}, Graduate Texts in Mathematics,
Springer, New York, 1976.

\bibitem{I}
S. Illman,
Smooth equivariant triangulations of G-manifolds for G a finite group,
{\em Math. Ann.} {\bf 233} (1978) 199--220.

\bibitem{J}
C. Jordan, M\'emoire sur les \'equations diff\'erentielles lin\'eaires \`a int\'egrale alg\'ebrique,
{\em J. Reine Angew. Math.} {\bf 84} (1878) 89--215.

\bibitem{L} S.D. Liao, A theorem on periodic transformations of
    homology spheres, {\em Ann. of Math.} {\bf 56} (1952)
    68--83.

\bibitem{MS}
L. N. Mann, J. C. Su, Actions of elementary p-groups on manifolds,
{\em Trans. Amer. Math. Soc.} {\bf 106} (1963), 115--126.

\bibitem{MY}
W.H. Meeks, S.-T. Yau, Group actions on $\RR^3$,
{\em The Smith conjecture} (New York, 1979), 167--179,
Pure Appl. Math., 112, Academic Press (1984).

\bibitem{MZ}
S. Meng, D.-Q. Zhang, Jordan property for non-linear algebraic groups and
projective varieties, {\em Amer. J. Math.} {\bf 140} (2018), no. 4, 1133--1145.




\bibitem{Mi}
H. Minkowski, Zur Theorie der positiven quadratischen Formen,
{\em Journal f\"ur die reine und angewandte Mathematik} {\bf 101} (1887),
196–202. (See also Collected Works. I,
212–218, Chelsea Publishing Company, 1967.)


\bibitem{M0}
I. Mundet i Riera,
Jordan's theorem for the diffeomorphism group of some manifolds,
Proc. AMS {\bf 138} (2010) 2253--2262.

\bibitem{M2}
I. Mundet i Riera,
Finite group actions on $4$-manifolds with nonzero Euler
characteristic, {\it Mathematische Zeitschrift} {\bf 282} (2016) 25--42,
DOI 10.1007/s00209-015-1530-8.

\bibitem{M6} I. Mundet i Riera,
Non Jordan groups of diffeomorphisms and actions of compact Lie groups on manifolds,
{\it Transformation Groups} {\bf 22} (2017), no. 2, 487--501,
DOI 10.1007/s00031-016-9374-9.

\bibitem{M5} I. Mundet i Riera,
Finite groups acting symplectically on $T^2\times S^2$,
{\it Trans. Amer. Math. Soc.} {\bf 369} (2017) 4457--4483.

\bibitem{M7} I. Mundet i Riera, Finite subgroups of Ham and
    Symp, {\em Math. Annalen} {\bf 370} (2018) 331--380,
    DOI 10.1007/s00208-017-1566-7.

\bibitem{M8} I. Mundet i Riera,
Almost fixed points of finite group actions on manifolds without odd cohomology,
{\em preprint} {\tt arXiv:1805.02582}, to appear in {\em Transformation Groups}.

\bibitem{MT}
I. Mundet i Riera, A. Turull,
Boosting an analogue of Jordan's theorem for finite groups,
{\em Adv. Math.} {\bf 272} (2015), 820--836.

\bibitem{Pe0}
T. Petrie,
Free metacyclic group actions on homotopy spheres,
{\em Ann. of Math.} (2) {\bf 94} (1971), 108–124.

\bibitem{PR}
T. Petrie, J.D. Randall, Finite-order algebraic automorphisms of affine varieties,
{\em Comment. Math. Helvetici} {\bf 61} (1986), 200--221.

\bibitem{Po0} V.L. Popov, On the Makar-Limanov, Derksen
    invariants, and finite automorphism groups of algebraic
    varieties. In {\em Peter Russell's Festschrift, Proceedings of
    the conference on Affine Algebraic Geometry held in
    Professor Russell's honour}, 1–5 June 2009, McGill Univ.,
    Montreal., volume 54 of Centre de Recherches Math\'ematiques
    CRM Proc. and Lect. Notes, pages 289–311, 2011.

\bibitem{Po}
V.L. Popov, Finite subgroups of diffeomorphism groups,
{\em Proc. Steklov Inst. Math.} {\bf 289} (2015), no. 1, 221--226.


\bibitem{Po2} V.L. Popov, Jordan groups and automorphism groups
    of algebraic varieties, in Cheltsov et al. (ed.),
    {\em Automorphisms in Birational and Affine Geometry},
    Springer Proceedings in Mathematics and Statistics {\bf
    79},
    Springer (2014), 185--213.

\bibitem{Po3} V.L. Popov, The Jordan property for Lie groups and automorphism groups of complex spaces,
{\em Math. Notes} {\bf 103} (2018), no. 5-6, 811--819.
 
\bibitem{PS}
Y. Prokhorov, C. Shramov, Jordan property for groups of birational selfmaps,
{\em Compos. Math.} {\bf 150} (2014), no. 12, 2054--2072.

\bibitem{PS2}
Y. Prokhorov, C. Shramov, Automorphism groups of compact complex surfaces,
{\it preprint} {\tt arXiv:1708.03566}.

\bibitem{Sch} R. Schultz,
Nonlinear analogs of linear group actions on spheres,
{\em Bull. Amer. Math. Soc.} {\bf 11} (1984), 263--285.

\bibitem{S3} J.--P. Serre, Le groupe de Cremona et ses sous-groupes finis, S\'eminaire Bourbaki,
no. 1000, Novembre 2008, 24 pp.

\bibitem{S2} J.--P. Serre, A Minkowski-style bound for the
    orders of the finite subgroups of the Cremona group of rank
    $2$ over an arbitrary field, {\em Moscow Mathematical
    Journal} {\bf 9} (2009), 193--208.

\bibitem{Ye}
S. Ye, Symmetries of flat manifolds, Jordan property and the general Zimmer program,
    {\em preprint} {\tt  arXiv:1704.03580}.

\bibitem{Ye2}
S. Ye, Euler characteristics and actions of automorphism groups of free groups,
{\em Algebr. Geom. Topol.} {\bf 18} (2018), 1195--1204.


\bibitem{Z-1} B.P. Zimmermann,
On the minimal dimension of a homology sphere on which a finite group acts,
{\em Math. Proc. Cambridge Philos. Soc.} {\bf 144} (2008), no. 2, 397--401.


\bibitem{Z-2}
B.P. Zimmermann,
On minimal actions of finite simple groups on homology spheres and Euclidean spaces,
{\em Rend. Circ. Mat. Palermo} (2) {\bf 59} (2010), no. 3, 451--459.


\bibitem{Z}
B.P. Zimmermann,
On finite groups acting on spheres and finite subgroups of orthogonal groups,
{\em Sib. Elektron. Mat. Izv.} 9 (2012), 1--12.

\bibitem{Z2}
B.P. Zimmermann, On Jordan
    type bounds for finite
    groups acting on compact $3$-manifolds,
    {\em Arch. Math.} {\bf 103} (2014), 195--200.

\bibitem{Z2016}
B.P. Zimmermann,
    A note on topological actions of finite, non-standard groups on spheres,
    {\em preprint} {\tt arXiv:1602.04599}.

\end{thebibliography}
\end{document}